\documentclass[12pt]{amsart}

\usepackage{amsmath,amssymb,epic,lscape}

\newtheorem{theorem}{Theorem}[section]
\newtheorem{proposition}[theorem]{Proposition}

\theoremstyle{definition}
\newtheorem{definition}[theorem]{Definition}

\theoremstyle{remark}

\numberwithin{equation}{section}
\providecommand{\bysame}{\leavevmode\hbox to3em{\hrulefill}\thinspace}

\def\DJ{{\hbox{D\kern-.8em\raise.15ex\hbox{--}\kern.35em}}}
\def\DJo{$\;$\kern-.4em
    \hbox{D\kern-.8em\raise.15ex\hbox{--}\kern.35em okovi\'{c}}}

\def\Gr{{ G_{\rm NS} }}
\def\al{{\alpha}}

\def\sig{{\sigma}}

\def\bZ{{\mbox{\bf Z}}}

\def\pE{{\mathcal E}}

\renewcommand{\subjclassname}{\textup{2000} Mathematics Subject
Classification }

\begin{document}

\title[Classification of normal sequences]
{Classification of normal sequences}

\author[D.\v{Z}. \DJ okovi\'{c}]
{Dragomir \v{Z}. \DJ okovi\'{c}}

\address{Department of Pure Mathematics and Institute for Quantum Computing, 
University of Waterloo, Waterloo, Ontario, N2L 3G1, Canada}

\email{djokovic@uwaterloo.ca}



\date{}

\begin{abstract}
Base sequences $BS(m,n)$ are quadruples $(A;B;C;D)$ of $\{\pm1\}$-sequences,
with $A$ and $B$ of length $m$ and $C$ and $D$ of length $n$, 
such that the sum of their nonperiodic autocorrelation 
functions is a $\delta$-function. Normal sequences $NS(n)$ are
base sequences $(A;B;C;D)\in BS(n,n)$ such that $A=B$.
We introduce a definition of equivalence for normal sequences 
$NS(n)$, and construct a canonical form.
By using this canonical form, we have enumerated the equivalence 
classes of $NS(n)$ for $n \le 40$. 
\end{abstract}

\maketitle
\subjclassname{ 05B20, 05B30 }
\vskip5mm

\section{Introduction} \label{Uvod}

By a {\em binary} respectively {\em ternary sequence} we mean a 
sequence $A=a_1,a_2,\ldots,a_m$ whose terms belong to $\{\pm1\}$ 
respectively $\{0,\pm1\}$. 
To such a sequence we associate the polynomial 
$A(z)=a_1+a_2z+\cdots+a_mz^{m-1}$.
We refer to the Laurent polynomial $N(A)=A(z)A(z^{-1})$ as 
the {\em norm} of $A$.
{\em Base sequences} $(A;B;C;D)$ are quadruples of binary sequences,
with $A$ and $B$ of length $m$ and $C$ and $D$ of length $n$, 
and such that 
\begin{equation} \label{norm}
N(A)+N(B)+N(C)+N(D)=2(m+n).
\end{equation}
The set of such sequences will be denoted by $BS(m,n)$.

In this paper we consider only the case where $m=n$ or $m=n+1$. 
The base sequences $(A;B;C;D)\in BS(n,n)$ are {\em normal} if $A=B$.
We denote by $NS(n)$ the set of normal sequences of length $n$,
i.e., those contained in $BS(n,n)$. It is well known \cite{Y} 
that for normal sequences $2n$ must be a sum of three squares. 
In particular, $NS(14)$ and $NS(30)$ are empty. Exhaustive 
computer searches have shown that $NS(n)$ are empty also for
$n=6,17,21,22,23,24$ (see \cite{KSY}) and 
$n=27,28,31,33,34,\ldots,39$ (see \cite{DZ2,DZ4}).

The base sequences $(A;B;C;D)\in BS(n+1,n)$ are {\em near-normal} 
if $b_i=(-1)^{i-1}a_i$ for all $i\le n$.
For near-normal sequences $n$ must be even or 1.
We denote by $NN(n)$ the set of near-normal 
sequences in $BS(n+1,n)$.

Normal sequences were introduced by C.H. Yang in \cite{Y}
as a generalization of Golay sequences. Let us recall 
that {\em Golay sequences} $(A;B)$ are pairs of binary sequences
of the same length, $n$, and such that $N(A)+N(B)=2n$.
We denote by $GS(n)$ the set of Golay sequences of length $n$.
It is known that they exist when $n=2^a 10^b 26^c$ where $a,b,c$ are
arbitrary nonnegative integers.
There exist two embeddings $GS(n)\to NS(n)$: the first 
defined by $(A;B)\to(A;A;B;B)$ and the second by 
$(A;B)\to(B;B;A;A)$. We say that these normal sequences (and those
equivalent to them) are of {\em Golay type}. For the definition of 
equivalence of normal sequences see section \ref{Ekv}. 
However, as observed by Yang, there exists normal 
sequences which are not of Golay type.
We refer to them as {\em sporadic} normal sequences. From the 
computational results reported in this paper (see Table 1 below) it appears 
that there may be only finitely many sporadic normal sequences. 
E.g. all 304 equivalence classes in $NS(40)$ are of Golay type.
The smallest length for which the existence question of
normal sequences is still unresolved is $n=41$.

Base sequences, and their special cases such as normal and near-normal
sequences, play an important role in the construction of Ha\-da\-mard
matrices \cite{HCD,SY}. For instance, the discovery of a 
Ha\-da\-mard matrix of order 428 (see \cite{KT}) used a 
$BS(71,36)$, constructed specially for that purpose.

Examples of normal sequences $NS(n)$ have been constructed in
\cite{DZ3,DZ7,HCD,KSY,Y}. For various applications,
it is of interest to classify the normal sequences of
small length. Our main goal is to provide such classification
for $n\le40$. The classification of near-normal sequences $NN(n)$ 
for $n\le40$ and base sequences $BS(n+1,n)$ for $n\le30$ has been carried out in our papers \cite{DZ3,DZ4,DZ6} and \cite{DZ7}, 
respectively.

We give examples of normal sequences of lengths $n=1,\ldots,5$:

\begin{center}
$$ \begin{array}{llll}
\begin{array}{l}
A=+; \\
A=+; \\
C=+; \\
D=+; 
\end{array} &
\begin{array}{l}
A=+,+; \\
A=+,+; \\
C=+,-; \\
D=+,-; 
\end{array} &
\begin{array}{l}
A=+,+,-; \\
A=+,+,-; \\
C=+,+,+; \\
D=+,-,+; 
\end{array} &
\begin{array}{l}
A=+,+,-,+; \\
A=+,+,-,+; \\
C=+,+,+,-; \\
D=+,+,+,-; 
\end{array}
\end{array} $$
\end{center}

$$ \begin{array}{l}
A=+,+,+,-,+; \\
A=+,+,+,-,+; \\
C=+,+,+,-,-; \\
D=+,-,+,+,-; 
\end{array} $$
When displaying a binary sequence, we often write $+$ for $+1$ 
and $-$ for $-1$. We have written the sequence $A$ twice to 
make the quads visible (see the next section).

If $(A;A;C;D)\in NS(n)$ then $(A,+;A,-;C;D)\in BS(n+1,n)$. 
This has been used in our previous papers to view normal sequences 
$NS(n)$ as a subset of $BS(n+1,n)$. For classification purposes it is 
more convenient to use the definition of $NS(n)$ as a subset of
$BS(n,n)$, which is closer to Yang's original definition \cite{Y}.

In section \ref{NorNiz} we recall  the basic properties of base 
sequences $BS(m,n)$. The quad decomposition and our encoding scheme 
for $BS(n+1,n)$ used in our previous papers also works for $NS(n)$, 
but not for arbitrary base sequences in $BS(n,n)$. The quad decomposition
of normal sequences $NS(n)$ is somewhat simpler than that of base 
sequences $BS(n+1,n)$. We warn the reader that the encodings for the 
first two sequences of $(A;A;C;D)\in NS(n)$ and 
$(A,+;A,-;C;D)\in BS(n+1,n)$ are quite different.

In section \ref{Ekv} we introduce the elementary transformations of 
$NS(n)$. We point out that the elementary transformation (E4) is 
quite non-intuitive. It originated in our paper \cite{DZ3} where 
we classified near-normal sequences of small length. Subsequently 
it has been extended and used to classify (see \cite{DZ7}) the base sequences $BS(n+1,n)$ for $n\le30$.  We use these elementary 
transformations to define an equivalence relation and equivalence 
classes in $NS(n)$. We also introduce the canonical form for normal 
sequences, and by using it we were able to compute the 
representatives of the equivalence classes for $n\le40$.

In section \ref{Grupa} we introduce an abstract group, $\Gr$, 
of order $512$ which acts naturally on all sets $NS(n)$. 
Its definition depends on the parity of $n$. The orbits of this 
group are just the equivalence classes of $NS(n)$.

In section \ref{Tablice} we tabulate the results of our
computations giving the list of representatives of the equivalence 
classes of $NS(n)$ for $n\le40$. The representatives are 
written in the encoded form which is explained in the next section.

The summary is given in Table 1. The column ``Equ'' gives the number 
of equivalence classes in $NS(n)$. Note that most of the known 
normal sequences are of Golay type. The column ``Gol'' respectively
``Spo'' gives the number of equivalence classes which are of
Golay type respectively sporadic. (Blank entries are zeros.)

\begin{center}
Table 1: Number of equivalence classes of $NS(n)$
\begin{tabular}{rrrrcrrrr} \\ \hline 
\multicolumn{1}{c}{$n$} & \multicolumn{1}{c}{Equ} 
& \multicolumn{1}{c}{Gol} & \multicolumn{1}{c}{Spo} 
& \multicolumn{1}{c}{\quad\quad\quad} 
& \multicolumn{1}{c}{$n$} & \multicolumn{1}{c}{Equ} 
& \multicolumn{1}{c}{Gol} & \multicolumn{1}{c}{Spo} \\ \hline
1 & 1 & 1 &   && 21 &   &&\\
2 & 1 & 1 &   && 22 &   &&\\
3 & 1 &   & 1 && 23 &   &&\\
4 & 1 & 1 &   && 24 &   &&\\
5 & 1 &   & 1 && 25 & 4 &   & 4 \\
6 &   &   &   && 26 & 2 & 2 &   \\
7 & 4 &   & 4 && 27 &   &&\\
8 & 7 & 6 & 1 && 28 &   &&\\
9 & 3 &   & 3 && 29 & 2 & & 2 \\
10 & 5 & 4 & 1 && 30 &  &&\\
11 & 2 &   & 2 && 31 &  &&\\
12 & 4 &   & 4 && 32 & 516 & 480 & 36 \\
13 & 3 &   & 3 && 33 &  &&\\
14 &   &   &   && 34 &  &&\\
15 & 2 &   & 2 && 35 &  &&\\
16 & 52 & 48 & 4 && 36 & &&\\
17 &   & & && 37 &  &&\\
18 & 1 &  & 1 && 38 & &&\\
19 & 1 &  & 1 && 39 & &&\\
20 & 36 & 34 & 2 && 40 & 304 & 304 & \\
\hline
\end{tabular} \\
\end{center}

\section{Quad decomposition and the encoding scheme} \label{NorNiz}

Let $ A=a_1,a_2,\ldots,a_n $ be an integer sequence of length $n$.
To this sequence we associate the polynomial
$$ A(x)=a_1+a_2x+\cdots+a_nx^{n-1} , $$
viewed as an element of the Laurent polynomial ring 
$\bZ[x,x^{-1}]$. (As usual, $\bZ$ denotes the ring of integers.)
The {\em nonperiodic autocorrelation function} $N_A$ of $A$ is 
defined by:
$$ N_A(i)=\sum_{j\in\bZ} a_ja_{i+j},\quad i\in\bZ, $$
where $a_k=0$ for $k<1$ and for $k>n$. Note that
$N_A(-i)=N_A(i)$ for all $i\in\bZ$ and $N_A(i)=0$ for $i\ge n$.
The {\em norm} of $A$ is the Laurent polynomial 
$N(A)=A(x)A(x^{-1})$. We have
$$ N(A)=\sum_{i\in\bZ} N_A(i) x^i . $$
Hence, if $(A;B;C;D)\in BS(m,n)$ then
\begin{equation} \label{KorNula}
N_A(i)+N_B(i)+N_C(i)+N_D(i)=0, \quad i\ne0.
\end{equation}

The negation, $-A$, of $A$ is the sequence
$$ -A=-a_1,-a_2,\ldots,-a_n. $$
The {\em reversed} sequence $A'$ and the {\em alternated} sequence
$A^*$ of the sequence $A$ are defined by
\begin{eqnarray*}
A' &=& a_n,a_{n-1},\ldots,a_1 \\
A^* &=& a_1,-a_2,a_3,-a_4,\ldots,(-1)^{n-1}a_n.
\end{eqnarray*}
Observe that $N(-A)=N(A')=N(A)$ and $N_{A^*}(i)=(-1)^i N_A(i)$
for all $i\in\bZ$. By $A,B$ we denote the concatenation of the 
sequences $A$ and $B$.

Let $(A;A;C;D) \in NS(n)$. 
For convenience we set $n=2m$ ($n=2m+1$) for $n$ even (odd).
We decompose the pair $(C;D)$ into quads
$$ \left[ \begin{array}{ll} c_i & c_{n+1-i} \\ 
d_i & d_{n+1-i} \end{array} \right],\quad i=1,2,\ldots,m, $$ 
and, if $n$ is odd, the central column
$ \left[ \begin{array}{l} c_{m+1} \\ d_{m+1} \end{array} \right]. $
Similar decomposition is valid for the pair $(A;A)$.

The possibilities for the quads of base sequences $BS(n+1,n)$ are 
described in detail in \cite{DZ7}. In the case of normal sequences
we have 8 possibilities for the quads of $(C;D)$:
\begin{center}
\begin{eqnarray*}
1=\left[ \begin{array}{ll} + & + \\ + & + \end{array} \right],\quad 
2=\left[ \begin{array}{ll} + & + \\ - & - \end{array} \right],\quad 
3=\left[ \begin{array}{ll} - & + \\ - & + \end{array} \right],\quad 
4=\left[ \begin{array}{ll} + & - \\ - & + \end{array} \right], \\
5=\left[ \begin{array}{ll} - & + \\ + & - \end{array} \right],\quad 
6=\left[ \begin{array}{ll} + & - \\ + & - \end{array} \right],\quad 
7=\left[ \begin{array}{ll} - & - \\ + & + \end{array} \right],\quad 
8=\left[ \begin{array}{ll} - & - \\ - & - \end{array} \right],
\end{eqnarray*}
\end{center}
but only 4 possibilities , namely 1,3,6 and 8, for the quads of 
$(A;A)$. In \cite{DZ7} we referred to these eight quads as BS-quads. 
The additional eight Golay quads were also needed for the
classification of base sequences $BS(n+1,n)$. Unless stated
otherwise, the word ``quad'' will refer to BS-quads.

We say that a quad is {\em symmetric} if its two columns are
the same, and otherwise we say that it is {\em skew}. The 
quads $1,2,7,8$ are symmetric and $3,4,5,6$ are skew. We say that
two quads have the {\em same symmetry type} if they are both
symmetric or both skew.

There are 4 possibilities for the central column:
$$
0=\left[ \begin{array}{l} + \\ + \end{array} \right],\quad
1=\left[ \begin{array}{l} + \\ - \end{array} \right],\quad
2=\left[ \begin{array}{l} - \\ + \end{array} \right],\quad
3=\left[ \begin{array}{l} - \\ - \end{array} \right].
$$

We encode the pair $(A;A)$ by the symbol sequence
\begin{equation} \label{simb-p}
p_1p_2 \ldots p_m \quad  \text{respectively} 
\quad p_1p_2 \ldots p_m p_{m+1}
\end{equation}
when $n$ is even respectively odd. Here $p_i$ is the label of the 
$i$th quad for $i\le m$ and $p_{m+1}$ is the label of the central 
column (when $n$ is odd). Similarly, we encode the pair $(C;D)$ 
by the symbol sequence
\begin{equation} \label{simb-q}
q_1q_2 \ldots q_m \quad  \text{respectively} 
\quad q_1q_2 \ldots q_m q_{m+1}.
\end{equation}
For example, the five normal sequences displayed in the introduction
are encoded as $(0;\, 0)$, $(1;\, 6)$, $(60;\, 11)$, $(16;\, 61)$ 
and $(160;\, 640)$, respectively. 

\section{The equivalence relation} \label{Ekv}

We start by defining five types of {\em elementary 
transformations} of normal sequences $(A;A;C;D)\in NS(n)$:

(E1) Negate both sequences $A;A$ or one of $C;D$. 

(E2) Reverse both sequences $A;A$ or one of $C;D$.

(E3) Interchange the sequences $C;D$.

(E4) Replace the pair $(C;D)$ with the pair 
$(\tilde{C};\tilde{D})$ which is defined as follows:
If (\ref{simb-q}) is the encoding of $(C;D)$, then 
the encoding of $(\tilde{C};\tilde{D})$ is 
$\tau(q_1)\tau(q_2)\cdots\tau(q_m)$ or 
$\tau(q_1)\tau(q_2)\cdots\tau(q_m) q_{m+1}$
depending on whether $n$ is even or odd,
where $\tau$ is the transposition $(45)$.
In other words, the encoding of $(\tilde{C};\tilde{D})$ is 
obtained from that of $(C;D)$ by replacing simultaneously each 
quad symbol 4 with the symbol 5, and vice versa. For the proof of 
the equality $N_{\tilde{C}}+N_{\tilde{D}}=N_C+N_D$ see \cite{DZ7}.

(E5) Alternate all four sequences $A;A;C;D$.

We say that two members of $NS(n)$ are {\em equivalent} if 
one can be transformed to the other by applying a finite 
sequence of elementary transformations. 
One can enumerate the equivalence classes by finding suitable
representatives of the classes.
For that purpose we introduce the canonical form.

\begin{definition} \label{KanFor}
Let $S=(A;A;C;D) \in NS(n)$ and let (\ref{simb-p})
respectively (\ref{simb-q}) be the encoding of the pair
$(A;A)$ respectively $(C;D)$. 
We say that $S$ is in the {\em canonical form} if the following 
twelve conditions hold:

(i) For $n$ even $p_1=1$, and for $n>1$ odd $p_1\in\{1,6\}$.

(ii)  The first symmetric quad (if any) of $(A;A)$ is 1.

(iii) The first skew quad (if any) of $(A;A)$ is 6.

(iv) If $n$ is odd and all quads of $(A;A)$ are skew, 
then $p_{m+1}=0$.

(v) If $n$ is odd and $i<m$ is the smallest index such that
the consecutive quads $p_i$ and $p_{i+1}$ have the same 
symmetry type, then $p_{m+1}\in\{1,6\}$. If there is no such 
index and $p_m$ is symmetric, then $p_{m+1}=0$.

(vi) $q_1\in\{1,6\}$ if $n>1$.

(vii) The first symmetric quad (if any) of $(C;D)$ is 1.

(viii) The first skew quad (if any) of $(C;D)$ is 6.

(ix) If $i$ is the least index such that $q_i\in\{2,7\}$ 
then $q_i=2$.

(x) If $i$ is the least index such that 
$q_i\in\{4,5\}$ then $q_i=4$.

(xi) If $n$ is odd and $q_i\ne2$, $\forall i\le m$, 
then $q_{m+1}\ne2$.

(xii) If $n$ is odd and $q_i\ne1$, $i\le m$, then $q_{m+1}=0$.
\end{definition}

We can now prove that each equivalence class has a member 
which is in the canonical form. The uniqueness of this member 
will be proved in the next section.

\begin{proposition} \label{Klase}
Each equivalence class $\pE\subseteq NS(n)$ has at least 
one member having the canonical form.
\end{proposition}
\begin{proof}
Let $S=(A;A;C;D)\in\pE$ be arbitrary and let (\ref{simb-p}) 
respectively (\ref{simb-q}) be the encoding of $(A;A)$ 
respectively $(C;D)$. By applying the elementary transformations 
(E1), we can assume that $a_1=c_1=d_1=+1$. If $n=1$, $S$ is in the 
canonical form. So, let $n>1$ from now on. Note that now the first quads, $p_1$ and $q_1$, necessarily belong to $\{1,6\}$ and that 
$p_1\ne q_1$ by (\ref{KorNula}). In the case when $n$ is even 
and $p_1=6$ we apply the elementary transformation (E5). Note 
that (E5) preserves the quads $p_1$ and $q_1$. Thus the 
conditions (i) and (vi) for the canonical form are satisfied. 

The conditions (ii),(iii) and (iv) are pairwise disjoint, and so 
at most one of them may be violated. To satisfy (ii), it suffices 
(if necessary) to apply to the pair $(A;A)$ the transformation (E2). 
To satisfy (iii) or (iv), it suffices (if necessary) to apply to 
the pair $(A;A)$ the transformations (E1) and (E2).

For (v), assume that $p_i$ and $p_{i+1}$ have the same symmetry
type and that $i$ is the smallest such index. Also assume that 
$p_{i+1}\notin\{1,6\}$, i.e., $p_{i+1}\in\{3,8\}$. 

We first consider the case where $p_1=1$ and $p_i$ and $p_{i+1}$ 
are symmetric. By our assumption we have $p_{i+1}=8$ and, by 
the minimality of $i$, $i$ must be odd.
We first apply (E2) to the pair $(A;A)$ and then apply (E5).
The quads $p_j$ for $j\le i$ remain unchanged. On the other hand
(E2) fixes $p_{i+1}$ because it is symmetric, while (E5) replaces  
$p_{i+1}=8$ with 1 because $i+1$ is even. We have to make sure
that previously established conditions are not spoiled. 
Only condition (iii) may be affected. If so, we must have
$i=1$ and we simply apply (E2) again.

Next we consider the case where again $p_1=1$ while $p_i$ and
$p_{i+1}$ are now skew. Thus $p_{i+1}=3$ and $i$ is even.
We again apply (E2) to the pair $(A;A)$ and then apply (E5).
The quads $p_j$ for $j\le i$ again remain unchanged. On the 
other hand (E2) replaces $p_{i+1}=3$ with 6, while (E5) fixes 
it because $i+1$ is odd. Note that in this case none of the
conditions (i-iv) and (vi) will be spoiled.

The remaining two cases (where $p_1=6$) can be treated in a 
similar fashion. Now assume that any two consecutive quads
$p_i,p_{i+1}$ have different symmetry types and that
the last quad, $p_m$, is symmetric. Assume also that 
$p_{m+1}\ne0$, i.e., $p_{m+1}=3$. If $p_1=1$ then $m$ is
odd and we just apply (E5). Otherwise $p_1=6$ and $m$ is
even and we apply the elementary transformations (E1) and (E2) 
to the pair $(A;A)$ and then apply (E5).
After this change the conditions (i-vi) will be satisfied.

To satisfy (vii), in view of (vi) we may assume that $q_1=6$. 
If the first symmetric quad in $(C;D)$ is 2 respectively 7,
we reverse and negate $C$ respectively $D$. If it is 8, we reverse 
and negate both $C$ and $D$. Now the first symmetric quad will be 1.

To satisfy (viii), (if necessary) reverse $C$ or $D$, or both. 
To satisfy (ix), (if necessary) interchange $C$ and $D$. To satisfy 
(x), (if necessary) apply the elementary transformation (E4). 
Note that in this process we do not violate the previously established properties. 

To satisfy (xi), (if necessary) switch $C$ and $D$ and apply
(E4) to preserve (x). To satisfy (xii), (if necessary) replace 
$C$ with $-C'$ or $D$ with $-D'$, or both.

Hence $S$ is now in the canonical form.

\end{proof}

We end this section by a remark on Golay type normal sequences.
Let $(A;B)\in GS(n)$, with $n=2m>2$. While the Golay sequences
$(A;B)$ and $(B;A)$ are always considered as equivalent
(see \cite{DZ1}) the normal sequences $(A;A;B;B)$ and 
$(B;B;A;A)$ may be non-equivalent. It is easy to show that
in fact these two normal sequences are equivalent if and only 
if the binary sequences $A$ and $B^*$ are equivalent, i.e.,
if and only if $B^*\in\{A;-A;A';-A'\}$.

The equivalence classes of Golay sequences of length $\le 40$ 
have been enumerated in \cite{DZ1}. This was accomplished by
defining the canonical form and listing the canonical 
representatives of the equivalence classes. These representatives 
are written there in encoded form as 
$\delta_1\delta_2\cdots\delta_m$ obtained by decomposing
$(A;B)$ into $m$ quads. These are Golay quads and should not 
be confused with the BS-quads defined in section \ref{NorNiz}.
If $(A;B)\in GS(n)$ is one of the representatives, it is 
obvious that $B^*\ne -A$ and $B^*\ne -A'$, and it is easy to 
see that also $B^*\ne A$. Thus if $B^*$ is equivalent to $A$
we must have $B^*=A'$. Finally, one can show that the
equality $B^*=A'$ holds if and only if 
$\delta_i \equiv i \pmod{2}$ for each index $i$. For another 
meaning of the latter condition see 
\cite[Proposition 5.1]{DZ1}. Thus an equivalence class of
Golay sequences $GS(n)$ with canonical representative $(A;B)$ 
provides either one or two equivalence classes of $NS(n)$.
The former case occurs if and only if 
$\delta_i \equiv i \pmod{2}$ for each index $i$. 

By using this criterion it is straightforward to list the
equivalence classes of $NS(n)$ of Golay type for 
$n\le40$. For instance if $n=8$ there are five equivalence 
classes of Golay sequences. Their representatives are
(see \cite{DZ1}) 3218, 3236, 3254, 3272 and 3315. Only the
last representative violates the above condition. Hence
we have exactly $4+2=6$ equivalence classes of Golay type
in $NS(8)$.

\section{The symmetry group of $NS(n)$} \label{Grupa}

We shall construct a group $\Gr$ of order $512$ which acts on $NS(n)$. 
Our (redundant) generating set for $\Gr$ will consist of 9 involutions. 
Each of these generators is an elementary transformation, and we use 
this information to construct $\Gr$, i.e., to impose the defining 
relations. We denote by $S=(A;A;C;D)$ an aritrary member of $NS(n)$.

To construct $\Gr$, we start with an elementary abelian group $E$ 
of order $64$ with generators $\nu,\rho$, and $\nu_i,\rho_i$, $i\in\{3,4\}$.
It acts on $NS(n)$ as follows:
\begin{eqnarray*}
&& \nu S=(-A;-A;C;D),\quad \rho S=(A';A';C;D), \\
&& \nu_3S=(A;A;-C;D),\quad \rho_3S=(A;A;C';D), \\
&& \nu_4S=(A;A;C;-D),\quad \rho_4S=(A;A;C;D').
\end{eqnarray*}

Next we introduce the involutory generator $\sig$. 
We declare that $\sig$ commutes with $\nu$ and $\rho$, and that
$\sig\nu_3=\nu_4\sig$ and $\sig\rho_3=\rho_4\sig$.
The group $H=\langle E,\sig \rangle$ is the direct product of two groups:
$H_1= \langle \nu,\rho \rangle$ of order 4
and $H_2=\langle \nu_3,\rho_3,\sig \rangle$ of order 32.
The action of $E$ on $NS(n)$ extends to $H$ by defining
$\sig S=(A;A;D;C).$

We add a new generator $\theta$ which commutes elementwise 
with $H_1$, commutes with $\nu_3\rho_3,\nu_4\rho_4$ and $\sig$, 
and satisfies
$\theta\rho_3=\rho_4\theta$. Let us denote this enlarged group 
by $\tilde{H}$. It has the direct product decomposition 
$$ \tilde{H}=\langle H,\theta \rangle = H_1 \times \tilde{H}_2,$$
where the second factor is itself direct product of two 
copies of the dihedral group $D_8$ of order 8:
$$ \tilde{H}_2=\langle \rho_3,\rho_4,\theta \rangle \times
\langle \nu_3\rho_3,\nu_4\rho_4,\theta\sig \rangle. $$
The action of $H$ on $NS(n)$ extends to $\tilde{H}$ by 
letting $\theta$ act as the elementary transformation (E5).

Finally, we define $\Gr$ as the semidirect product of 
$\tilde{H}$ and the group of order 2 with generator $\al$. 
By definition, $\al$ commutes with $\nu,\nu_3,\nu_4$ and satisfies:
\begin{eqnarray*}
&& \al\rho\al=\rho(\nu\sig_1)^{n-1}; \\
&& \al\rho_j\al=\rho_j\nu_j^{n-1},\ j=3,4; \\
&& \al\theta\al=\theta\sig^{n-1}.
\end{eqnarray*}
The action of $ \tilde{H}$ on $NS(n)$ extends to $\Gr$ by 
letting $\al$ act as the elementary transformation (E5), i.e., 
we have $\al S=(A^*;B^*;C^*;D^*).$ 

We point out that the definition of the subgroup $\tilde{H}$ 
is independent of $n$ and its action on $NS(n)$ has a quad-wise 
character. By this we mean that the value of a particular 
quad, say $p_i$, of $S\in NS(n)$ and $h\in\tilde{H}$ determine 
uniquely the quad $p_i$ of $hS$. In other words
$\tilde{H}$ acts on the quads and the
set of central columns such that the encoding of $hS$ is
given by the symbol sequences
$$ h(p_1)h(p_2)\ldots  \quad \text{and} \quad
h(q_1)h(q_2)\ldots \, .$$
On the other hand the definition of the full group $\Gr$ depends 
on the parity of $n$, and only for $n$ odd it has the
quad-wise character. 

An important feature of the quad-action of $\tilde{H}$ 
is that it preserves the symmetry type of the quads. If $n$ 
is odd, this is also true for $\Gr$.

The following proposition follows immediately from the construction 
of $\Gr$ and the description of its action on $NS(n)$.

\begin{proposition} \label{Orbite}
The orbits of $\Gr$ in $NS(n)$ are the same as the
equivalence classes.
\end{proposition}

The main tool that we use to enumerate the equivalence classes 
of $NS(n)$ is the following theorem.

\begin{theorem} \label{Glavna}
For each equivalence class $\pE\subseteq NS(n)$ there is a 
unique $S=(A;A;C;D)\in\pE$ having the canonical form.
\end{theorem}
\begin{proof}
In view of Proposition \ref{Klase}, we just have to prove the 
uniqueness assertion. Let
$$
S^{(k)}=(A^{(k)};A^{(k)};C^{(k)};D^{(k)})\in\pE,\quad (k=1,2)
$$
be in the canonical form. We have to prove that in fact 
$S^{(1)}=S^{(2)}$.

By Proposition \ref{Orbite}, we have $gS^{(1)}=S^{(2)}$ for some 
$g\in\Gr$. We can write $g$ as $g=\al^s h$ where $s\in\{0,1\}$ and 
$h=h_1h_2$ with $h_1\in H_1$ and $h_2\in\tilde{H}_2$. 
Let $p_1^{(k)}p_2^{(k)}\ldots $ be 
the encoding of the pair $(A^{(k)};A^{(k)})$ and 
$q_1^{(k)}q_2^{(k)}\ldots$ the encoding of the pair 
$(C^{(k)};D^{(k)})$. The symbols (i-xii) will refer to the 
corresponding conditions of Definition \ref{KanFor}.

We prove first preliminary claims (a-c). 

(a): $p_1^{(1)}=p_1^{(2)}$ and, consequently, $q_1^{(1)}=q_1^{(2)}$. 

For $n$ even this follows from (i). Let $n$ be odd. 
When we apply the generator $\al$ to any $S\in NS(n)$, we 
do not change the first quad of $(A;A)$. It follows that the quads 
$p_1^{(1)}$ and 
$p_1^{(2)}=g\left(p_1^{(1)}\right)=h_1\left(p_1^{(1)}\right)$
have the same symmetry type. The claim now follows from (i).

Clearly, we are done with the case $n=2$.

If $n=3$ it is easy to see that we must have $p_1^{(1)}=p_1^{(2)}=6$ 
and $q_1^{(1)}=q_1^{(2)}=1$. By (iv), for the central column symbols, 
we have $p_2^{(1)}=p_2^{(2)}=0$. Then the equation (\ref{KorNula})
for $i=1$ implies that $q_2^{(k)}\in\{1,2\}$ for $k=1,2$. By (xi)
we must have  $q_2^{(1)}=q_2^{(2)}=1$. Hence $S^{(1)}=S^{(2)}$ in 
that case. 

Thus from now on we may assume that $n>3$.

(b): If $n$ is even then $s=0$. 

By (i), $p_1^{(1)}=p_1^{(2)}=1$. Note that the first quads of
$(A;A)$ in $S$ and in $\al S$ have different symmetry types
for any $S\in\pE$. As the quad $h(1)$ is symmetric, the equality 
$\al^s hS^{(1)}=S^{(2)}$ forces $s$ to be 0. 

As an immediate consequence of (b), we point out that, if $n$ is 
even, a quad $p_i^{(1)}$ is symmetric iff $p_i^{(2)}$ is,
and the same is true for the quads $q_i^{(1)}$ and $q_i^{(2)}$.

(c): $p_2^{(1)}=p_2^{(2)}$. 

We first observe that $p_2^{(1)}$ and $p_2^{(2)}$ have the same 
symmetry type. If $n$ is even this follows from (b) since then
$g=h$. If $n$ is odd then under the quad action on $p_2$,
each of $\al$, $\nu$, $\rho$ preserves the symmetry type of
$p_2$. Now the assertion (c) follows from (ii) and (iii) if
$p_1^{(1)}$ and $p_2^{(1)}$ have different symmetry types,
and from (v) otherwise.

We shall now prove that $A^{(1)}=A^{(2)}$. 

Assume first that $n$ is even. Then $p_1^{(1)}=p_1^{(2)}=1$ by (i), 
$s=0$ by (b), and the equality $h_1(p_1^{(1)})=p_1^{(2)}$ implies that $h_1(1)=1$. Thus $h_1\in\{1,\rho\}$. Let $i$ be the smallest 
index (if any) such that the quad $p_i^{(1)}$ is skew. Then
$p_i^{(1)}=p_i^{(2)}=6$ by (iii). Hence $h_1(6)=6$ and so 
$h_1=1$ and $A^{(1)}=A^{(2)}$ follows. On the other hand, if all
quads $p_i^{(1)}$ are symmetric, then all these quads are
fixed by $h_1$ and so $A^{(1)}=A^{(2)}$.

Next assume that $n$ is odd. Then $p_1^{(1)}=p_2^{(1)}\in\{1,6\}$
by (i). Let $i<m$ be the smallest index (if any) such that 
the quads $p_i^{(1)}$ and $p_{i+1}^{(1)}$ have the same 
symmetry type.

We first consider the case $p_1^{(1)}=1$. Since $n$ is odd
$\al$ fixes the quad $p_1$, and so $h_1$ must fix the quad 1.
Thus we again have $h_1\in\{1,\rho\}$. 

If $i$ is even then, by minimality 
of $i$, both $p_i^{(1)}$ and $p_{i+1}^{(1)}$ are skew. By (v) 
we have $p_{i+1}^{(1)}=p_{i+1}^{(2)}=6$. Since $i$ is even,
$\al$ fixes $p_{i+1}$ and so we must have $h_1(6)=6$. It follows
that $h_1=1$. As $i>1$, the quad $p_2^{(1)}$ is skew and
by (iii) we have $p_2^{(1)}=p_2^{(2)}=6$. Since $\al$ maps 
$p_2$ to its negative, we must have $s=0$. Consequently, 
$A^{(1)}=A^{(2)}$. 

If $i$ is odd then both $p_i^{(1)}$ and $p_{i+1}^{(1)}$ are 
symmetric. By (v) we have $p_{i+1}^{(1)}=p_{i+1}^{(2)}=1$. 
Since $i$ is odd, $\al$ maps $p_{i+1}$ to its negative. 
Since $\rho$ fixes the symmetric quads, we conclude that
$1=g(1)=\al^s h_1(1)=\al^s(1)$ and so $s=0$. If all quads 
$p_i^{(1)}$ are symmetric, then they are all fixed by $g$
and so $A^{(1)}=A^{(2)}$. Otherwise let $j$ be the smallest 
index such that $p_j^{(1)}$ is skew. By (iii) we have
$p_j^{(1)}=p_j^{(2)}=6$, and $6=p_j^{(2)}=g(p_j^{(1)})=g(6)=h_1(6)$ implies that $h_1=1$. Thus $A^{(1)}=A^{(2)}$. 

We now consider the case $p_1^{(1)}=6$. Since $n$ is odd
$\al$ fixes the quad $p_1$, and so $h_1$ must fix the quad 6.
Thus we have $h_1\in\{1,\nu\rho\}$. 

If $i$ is even then, by minimality 
of $i$, both $p_i^{(1)}$ and $p_{i+1}^{(1)}$ are symmetric. By (v) 
we have $p_{i+1}^{(1)}=p_{i+1}^{(2)}=1$. Since $i$ is even,
$\al$ fixes $p_{i+1}$ and so we must have $h_1(1)=1$. 
It follows that $h_1=1$. As $i>1$, the quad $p_2^{(1)}$ is 
symmetric and by (ii) we have $p_2^{(1)}=p_2^{(2)}=1$. Since 
$\al$ maps $p_2$ to its negative, we must have $s=0$. Consequently, 
$A^{(1)}=A^{(2)}$. 

If $i$ is odd then both $p_i^{(1)}$ and $p_{i+1}^{(1)}$ are 
skew. By (v) we have $p_{i+1}^{(1)}=p_{i+1}^{(2)}=6$. 
Since $i$ is odd, $\al$ maps $p_{i+1}$ to its negative. 
Since $\nu\rho$ fixes the skew quads, we conclude that
$6=g(6)=\al^s h_1(6)=\al^s(6)$ and so $s=0$. If all quads 
$p_i^{(1)}$, $i\le m$, are skew, then they are all fixed by $g$
and $p_{m+1}^{(1)}=p_{m+1}^{(2)}=0$ by (iv). Now 
$0=p_{m+1}^{(2)}=h_1(p_{m+1}^{(1)})=h_1(0)$ entails that
$h_1=1$ and so $A^{(1)}=A^{(2)}$. Otherwise let $j$ be the smallest 
index such that $p_j^{(1)}$ is symmetric. By (ii) we have
$p_j^{(1)}=p_j^{(2)}=1$, and $1=p_j^{(2)}=g(p_j^{(1)})=h_1(1)$ implies that $h_1=1$. Thus $A^{(1)}=A^{(2)}$. 

It remains to consider the case where any two consecutive 
quads $p_i^{(1)}$ and $p_{i+1}^{(1)}$, $i<m$, have different 
symmetry types. Say, the quads $p_i^{(1)}$, $i\le m$, 
are skew for even $i$ and symmetric for odd $i$. By (i) 
and (iii) we have $p_1^{(1)}=p_1^{(2)}=1$ and 
$p_2^{(1)}=p_2^{(2)}=6$. Then $h_1$ must fix the quad 1,
and so $h_1\in\{1,\rho\}$. Since 
$6=p_2^{(2)}=g(p_1^{(2)})=g(6)=\al^s h_1(6)$, we must have
$s=0$ and $h_1=1$ or $s=1$ and $h_1=\rho$. In the former 
case we obviously have $A^{(1)}=A^{(2)}$. In the latter case 
all quads $p_i^{(1)}$, $i\le m$, are fixed by $g$. Moreover,
if $m$ is even also the central column $p_{m+1}$
is fixed by $g$ and so $A^{(1)}=A^{(2)}$. On the other hand,
if $m$ is odd, then the quad $p_m^{(1)}$ is symmetric and
the second part of the condition (v) implies that
$p_{m+1}^{(1)}=p_{m+1}^{(2)}=0$. Hence again 
$A^{(1)}=A^{(2)}$. 

Similar proof can be used if the quads $p_i^{(1)}$, $i\le m$, 
are symmetric for even $i$ and skew for odd $i$. 
This completes the proof of the equality $A^{(1)}=A^{(2)}$. 
The proof of the equality $(C^{(1)};D^{(1)})=(C^{(2)};D^{(2)})$ 
is the same as in \cite{DZ3}.
\end{proof}

\section{Representatives of the equivalence classes} \label{Tablice}

We have computed a set of representatives for the equivalence
classes of normal sequences $NS(n)$ for all $n\le40$.  
Each representative is given in the canonical form
which is made compact by using our standard encoding.
The encoding is explained in detail in section \ref{NorNiz}. 
This compact notation is used primarily in order to save space, 
but also to avoid introducing errors during decoding. For each 
$n$, the representatives are listed in the lexicographic order 
of the symbol sequences (\ref{simb-p}) and (\ref{simb-q}).

In Table 2 and 3 we list the codes for the representatives of the
equivalence classes of $NS(n)$ for $n\le15$ and $16\le n\le29$, respectively. As there are 516 and 304 equivalence classes in 
$NS(32)$ and $NS(40)$ respectively, we list in Table 4 only 
the 36 representatives of the sporadic classes of $NS(32)$. 
The cases $$n=6,14,17,21,\ldots,24,27,28,30,31,33,34,\ldots,39$$ 
are omitted since then $NS(n)=\emptyset$. We also omit $n=40$ 
because in that case there are no sporadic classes.
The Golay type equivalence classes of normal sequences can be easily enumerated (as explained in section \ref{Ekv}) by using the tables 
of representatives of the equivalence classes of Golay 
sequences \cite{DZ1}.

\begin{center}
\begin{tabular}{rlrlrl}
\multicolumn{6}{c}{Table 2: Class representatives for $n\le15$} \\ \hline 
\multicolumn{6}{c}{$n=1$} \\
1 & 0 0 & &&&\\  \hline
\multicolumn{6}{c}{$n=2$} \\
1 & 6 1 & &&&\\  \hline
\multicolumn{6}{c}{$n=3$} \\
1 & 60 11 & &&&\\  \hline
\multicolumn{6}{c}{$n=4$} \\
1 & 16 61 & &&&\\  \hline
\multicolumn{6}{c}{$n=5$} \\
1 & 160 640 & &&&\\  \hline
\multicolumn{6}{c}{$n=7$} \\
1 & 1660 6122 & 2 & 6113 1623 & 3 & 6160 1262 \\
4 & 6163 1261 & &&& \\  \hline
\multicolumn{6}{c}{$n=8$} \\
1 & 1163 6618 & 2 & 1613 6168 & 3 & 1613 6443 \\
4 & 1638 6116 & 5 & 1661 6183 & 6 & 1686 6131 \\
7 & 1866 6311 & &&&\\  \hline
\multicolumn{6}{c}{$n=9$} \\
1 & 16133 64140 & 2 & 16163 64150 & 3 & 61180 16640 \\ \hline
\multicolumn{6}{c}{$n=10$} \\
1 & 11863 66311 & 2 & 16166 64156 & 3 & 16613 61838 \\
4 & 16616 61831 & 5 & 18863 63311 & &\\ \hline
\multicolumn{6}{c}{$n=11$} \\
1 & 611680 164231 & 2 & 616163 126232 && \\  \hline
\multicolumn{6}{c}{$n=12$} \\
1 & 161383 641261 & 2 & 163868 612243 & 3 & 186338 631422 \\
4 & 186631 631422 & &&&\\  \hline
\multicolumn{6}{c}{$n=13$} \\
1 & 1616133 6414853 & 2 & 6116680 1286320 & 3 & 6168160 1613441 \\ 
\hline
\multicolumn{6}{c}{$n=15$} \\
1 & 61613163 12676761 & 2 & 61683860 12626262 & &\\  \hline
\end{tabular} \\
\end{center}

Note that in the case $n=1$ there are no quads and both zeros 
in Table 2 represent central columns.

\begin{center}
\begin{tabular}{rlrl}
\multicolumn{4}{c}{Table 3: Class representatives for $16\le n\le29$} \\ \hline 
\multicolumn{4}{c}{$n=16$} \\
1 & 11186366 66631811 & 2 & 11186636 66631181 \\
3 & 11631866 66186311 & 4 & 11633381 66181163 \\
5 & 11636618 66188836 & 6 & 11638133 66183688 \\
7 & 11661836 66116381 & 8 & 11663681 66111863 \\
9 & 11666318 66118136 & 10 & 11668163 66113618 \\
11 & 11816333 66361888 & 12 & 11816663 66361118 \\
13 & 16131686 61686131 & 14 & 16133831 61681613 \\
15 & 16136168 61688386 & 16 & 16138313 61683868 \\
17 & 16161386 61616831 & 18 & 16163861 61611683 \\
19 & 16163861 64124328 & 20 & 16166138 61618316 \\
21 & 16166138 64127156 & 22 & 16168613 61613168 \\
23 & 16381331 61166813 & 24 & 16381661 61166183 \\
25 & 16388338 61163816 & 26 & 16388668 61163186 \\
27 & 16611368 61836886 & 28 & 16611638 61836116 \\
29 & 16618361 61833883 & 30 & 16618631 61833113 \\
31 & 16831313 61386868 & 32 & 16833838 61381616 \\
33 & 16836161 61384242 & 34 & 16836161 61388383 \\
35 & 16838686 61383131 & 36 & 16838863 61344313 \\
37 & 16861613 61316168 & 38 & 16863868 61311686 \\
39 & 16866131 61318313 & 40 & 16868386 61313831 \\
41 & 18116333 63661888 & 42 & 18116663 63661118 \\
43 & 18631133 63186688 & 44 & 18633388 63181166 \\
45 & 18636611 63188833 & 46 & 18638866 63183311 \\
47 & 18661163 63116618 & 48 & 18663688 63111866 \\
49 & 18666311 63118133 & 50 & 18668836 63113381 \\
51 & 18886366 63331811 & 52 & 18886636 63331181 \\ \hline
\multicolumn{4}{c}{$n=18$} \\
1 & 161633881 641242146 & &\\ \hline
\multicolumn{4}{c}{$n=19$} \\
1 & 1168186360 6643551210 & &\\ \hline
\multicolumn{4}{c}{$n=20$} \\
1 & 1166131836 6611686381 & 2 & 1166861836 6611316381 \\
3 & 1181616633 6636161188 & 4 & 1186161633 6631616188 \\
5 & 1186868366 6631313811 & 6 & 1188686366 6633131811 \\
7 & 1611663138 6441827614 & 8 & 1613383113 6168161368 \\
9 & 1613383186 6168161331 & 10 & 1616138631 6164224786 \\
11 & 1616311386 6161866831 & 12 & 1616681386 6161136831 \\
13 & 1616831361 6161386883 & 14 & 1616833886 6161381631 \\
15 & 1616836113 6161388368 & 16 & 1616838638 6161383116 \\
\hline
\end{tabular} \\
\end{center}

\begin{center}
\begin{tabular}{rlrl}
\multicolumn{4}{c}{Table 3: (continued)} \\ \hline 
\multicolumn{4}{c}{$n=20$} \\
17 & 1638133138 6116681316 & 18 & 1638133161 6116681383 \\
19 & 1638883818 6183331633 & 20 & 1661813881 6116361666 \\
21 & 1661863138 6183311316 & 22 & 1661863161 6183311383 \\
23 & 1683381313 6138836868 & 24 & 1683611313 6138166868 \\
25 & 1683831361 6138386883 & 26 & 1683833886 6138381631 \\
27 & 1683836113 6138388368 & 28 & 1683838638 6138383116 \\
29 & 1686613113 6131831368 & 30 & 1686613186 6131831331 \\
31 & 1863161133 6318616688 & 32 & 1863831133 6318386688 \\
33 & 1881616663 6336161118 & 34 & 1886161663 6331616118 \\
35 & 1886868336 6331313881 & 36 & 1888686336 6333131881 \\ \hline
\multicolumn{4}{c}{$n=25$} \\
1 & 1616138313163 & & 6414148485143 \\
2 & 1616161383163 & & 6414148584143 \\
3 & 1616161386163 & & 6414148585143 \\
4 & 1616168613163 & & 6414158585143 \\ \hline
\multicolumn{4}{c}{$n=29$} \\
1 & 161383131316830 & & 641414841515843 \\
2 & 161686161313860 & & 641515851514853 \\
\hline
\end{tabular} \\
\end{center}

\begin{center}
\begin{tabular}{rl}
\multicolumn{2}{c}{Table 4: Sporadic classes for $n=32$} \\ \hline 
1 & 1111636366331881 6666181845542277 \\
2 & 1111663318816363 6666455411882727 \\
3 & 1166186333886318 6641231814721176 \\
4 & 1166186366113681 6641231858635567 \\
5 & 1166813633883681 6614328141271167 \\
6 & 1166813666116318 6614328185365576 \\
7 & 1613161361683831 6168616842525747 \\
8 & 1616168313861313 6412651765826487 \\
9 & 1616168338613838 6412623728284126 \\
10 & 1616168361386161 6412623756567358 \\
11 & 1616383883163861 6412214634822843 \\
12 & 1616386113133168 6412434384672376 \\
13 & 1616386186866831 6412282832157623 \\
14 & 1616613813136831 6412565684677623 \\
15 & 1616613886863168 6412717132152376 \\
16 & 1616616116833861 6412785365172843 \\
17 & 1616831613868686 6412348265823512 \\
18 & 1616831638616161 6412376243437358 \\
19 & 1616831661383838 6412376271714126 \\
20 & 1638163886681331 6142241631477413 \\
21 & 1638163886681331 6241142632488423 \\
22 & 1661166113688631 6142758368527413 \\
23 & 1661166113688631 6241857367518423 \\
24 & 1683161638383861 6138642142161717 \\
25 & 1683161661616138 6138642183575656 \\
26 & 1683383813863131 6138421671711253 \\
27 & 1683383886136868 6138164234348746 \\
28 & 1683616113866868 6138428321218256 \\
29 & 1683616186133131 6138834235351743 \\
30 & 1683838338386138 6138342816574646 \\
31 & 1683838361613861 6138342842831212 \\
32 & 1686168638686131 6131613142475752 \\
33 & 1818633611886666 6363445518812222 \\
34 & 1818666636638811 6363111144552772 \\
35 & 1863116636816611 6341268841334537 \\
36 & 1863116663183388 6341268814221826 \\ \hline
\end{tabular} \\
\end{center}

\section{Acknowledgments}

The author is grateful to NSERC for the 
continuing support of his research. This work was made possible by 
the facilities 
of the Shared Hierarchical Academic Research Computing Network 
(SHARCNET:www.sharcnet.ca) and Compute/Calcul Canada.

\end{document}